\documentclass[a4paper,10pt]{article}
\usepackage[latin1]{inputenc}

\usepackage{lipsum}
\usepackage[dvipsnames]{xcolor}
\usepackage{amsmath}
\usepackage{amssymb}
\usepackage{bm}
\usepackage{tikz-network}
\usepackage{tikz}
\usepackage{caption}
\usepackage{subcaption}
\usepackage{float}
\usepackage[section]{placeins}
\usepackage{amsfonts}
\usepackage{graphicx}
\usepackage{amsthm}
\usepackage{wrapfig}
\usepackage{latexsym}
\usepackage{xspace}
\usepackage{hyperref}
\hypersetup{
    colorlinks=true, %Pour que les trucs cliquables soient colorés
    linktoc=all,     %pour que les sections soient cliquables
    linkcolor=blue,  %la couleur des trucs cliquables
}

\usepackage{graphicx}
\newcommand{\cev}[1]{\reflectbox{\ensuremath{\vec{\reflectbox{\ensuremath{#1}}}}}}

\newcommand*\Fin[0]
{{\mathcal{F}_{-}}}

\newcommand*\Fout[0]
{{\mathcal{F}_{+}}}

\newcommand*\Fc[0]
{\mathcal{F}}

\newcommand*\din[1]
{d_F^-(#1)}

\newcommand*\dind[1]
{d_A^-(#1)}

\newcommand*\dout[1]
{d_F^+(#1)}

\newcommand*\doutd[1]
{d_A^+(#1)}

\newcommand*\E[1]
{\mathcal{M}( #1)}

\newcommand*\Ein[1]
{{\mathcal{M}_{-}}( #1)}

\newcommand*\Eout[1]
{{\mathcal{M}_{+}}( #1)}

\newcommand*\Uin[1]
{{I}( #1)}

\newcommand*\Uout[1]
{{O}( #1)}

\newtheorem{theo}{Theorem}

\newtheorem{claim}{Claim}
\newtheorem{lemma}{Lemma}

\newtheorem{cor}{Corollary}

\title{On reversing arcs to improve arc-connectivity}
\author{Pierre Hoppenot, Zolt\'an Szigeti\\ \\
Univ.~Grenoble~Alpes, Grenoble INP, CNRS, G-SCOP.  }

\begin{document}

\maketitle

\begin{abstract}
 We show that if the 	arc-connectivity of a directed graph $D$ is at most $\lfloor\frac{k+1}{2}\rfloor$ and the reorientation of an arc set $F$ in $D$ results in a $k$-arc-connected directed graph then we can reorient one arc of $F$ without decreasing the arc-connectivity of $D.$ This improves a result of Fukuda,  Prodon, Sakuma \cite{FPS} and one of  Ito et al. \cite{9jap} for $k\in\{2,3\}$.
\end{abstract}

\section{Introduction}

In this paper we study reorientations of directed graphs that improve their arc-connectivity. Nash-Williams \cite{NW} proved that a directed graph $D$ admits a $k$-arc-connected reorientation if and only if the underlying undirected graph of $D$ is $2k$-edge-connected. Ito et al. \cite{9jap} proposed an original algorithm to find such an orientation when it exists, by starting with an arbitrary orientation and then by reversing arcs one by one, without decreasing the arc-connectivity of the current directed graph in each step. This way they provided a new algorithmic proof for the above mentioned result of Nash-Williams \cite{NW}. A natural question arises here: If $D$ is not $k$-arc-connected and the reorientation of an arc set $F$ in $D$ results in a $k$-arc-connected directed graph, is it possible to reorient the arcs of $F$ one by one, without decreasing the arc-connectivity of the current directed graph in each step. For $k=1,$ it is obvious and for $k=2,$ it follows from a result of Fukuda, Prodon, and Sakuma \cite{FPS} that this can be done. We prove that for $k=3$, the answer is positive and we show that the answer is negative for $k\ge 4.$  The result for $k\in\{2,3\}$ follows from  the main contribution of the present paper that  says  if the arc-connectivity  of a directed graph $D$ is at most $\lfloor\frac{k+1}{2}\rfloor$ and the reorientation of an arc set $F$ in $D$ results in a $k$-arc-connected directed graph then we can reorient one arc of $F$ without decreasing the arc-connectivity of $D.$ We also give an example that shows that in the previous result $\lfloor\frac{k+1}{2}\rfloor$ can not be replaced by $\lfloor\frac{2k}{3}\rfloor.$
\medskip

We finish this section with the necessary definitions. A set $\mathcal{L}$ of subsets of $V$ is called {\it  laminar} if for every pair $X,Y\in\mathcal{L}$, at least one of $X-Y, Y-X, X\cap Y$ is empty. As usually, for $X\subseteq V,$ {\boldmath$\overline{X}$} denotes $V-X.$
\smallskip

An undirected graph $G=(V,E)$ is {\it  $k$-edge-connected} if for every non-empty proper  subset $X$ of $V$, there exist at least $k$ edges leaving $X.$ An orientation of $G$ is a digraph which is obtained from $G$ by replacing each edge $uv$ in $G$ by an arc $uv$ or $vu.$
\smallskip

Let $D=(V,A)$ be a directed graph, in short {\it  digraph}. A digraph $D'=(V',A')$ is a {\it subgraph} of $D$ if  $V'\subseteq V$ and $A'\subseteq A.$ When $V'=V$ we call $D'$ a {\it spanning} subgraph. For $X\subseteq V,$ $uv\in A$ we say that $uv$ enters (resp. leaves) $X$ if $u\in V-X$ and $v\in X$ (resp. $u\in X$ and $v\in V-X$). We denote by {\boldmath$\rho_A(X)$} (resp. {\boldmath$\delta_A(X)$})  the set of arcs in $A$ that enter (resp. leave) $X$. The {\it  in-degree} $|\rho_A(X)|$ of $X$ is denoted by {\boldmath$d_A^-(X)$}. The {\it  out-degree} $|\delta_A(X)|$ of $X$ is denoted by {\boldmath$d_A^+(X)$}. For disjoint  $X,Y\subseteq V,$ {\boldmath$d_A(X,Y)$} denotes the number of arcs in $A$  from $X$ to $Y.$ We denote by {\boldmath$\lambda(D)$} $=\min\{d_A^-(X): \emptyset\neq X\subset V\}$ the {\it  arc-connectivity} of $D.$ We say that $D$ is {\it  $k$-arc-connected} if  $\lambda(D)\ge k.$ A subset $X$  of $V$ is {\it  in-tight} (resp. {\it  out-tight}) if $d^-_A(X)=\lambda(D)$ (resp. $d^+_A(X)=\lambda(D)$). {\it  Reversing}  an arc $\vec{e}=uv$ means that we replace $\vec{e}$ by {\boldmath$\cev{e}$} $=vu.$ $D(\cev{e})$ denotes the digraph obtained from $D$ by reversing $\vec{e}.$ The digraph obtained from $D$ by reversing all the arcs is denoted by {\boldmath$\cev{D}$}. A digraph $D'$ is called a {\it reorientation} of $D$ if $D'$ can be obtained from $D$ reversing some arcs of $D.$

\section{Previous results}

In this section we list some results on orientations and reorientations that are related to our results.
\medskip

The first result on orientations of undirected graphs concerning connectivity is due to Robbins \cite{robbins}.

\begin{theo}[Robbins \cite{robbins}]\label{Rthm}
An undirected graph admits a $1$-arc-connected orientation if and only if it is $2$-edge-connected.
\end{theo}

The following result of Fukuda,  Prodon, and Sakuma \cite{FPS} is about reversing arcs not decreasing $1$-arc-connectivity.

\begin{theo}[Fukuda,  Prodon,  Sakuma \cite{FPS}]\label{k2}
Let $D=(V,A)$ be a $1$-arc-connected orientation of a $3$-edge-connected undirected graph $G$. Let $F\subseteq A$ such that the digraph obtained from $D$ by reversing the arcs of $F$ is $1$-arc-connected. Then there exists an arc $\vec a$ of $F$ such that $D({\cev{a}})$ is $1$-arc-connected.
 \end{theo}

Theorem \ref{k2} immediately implies Corollary \ref{k2cor} concerning reversing arcs augmenting $1$-arc-connectivity.

\begin{cor}\label{k2cor}
Let $D_0=(V,A)$ be a digraph and $F\subseteq A$ such that  the digraph obtained from $D_0$ by reversing the arcs of $F$ is $2$-arc-connected. Then there exist reorientations $D_1, \dots,D_\ell$ of $D_0$ such that $D_{i+1}$ is obtained from $D_i$ by reversing one arc of $F$ for all $0\le i\le \ell-1$ and $1\le\lambda(D_1)\le \dots\le\lambda(D_\ell)=2.$
 \end{cor}

Nash-Williams \cite{NW} generalized Theorem \ref{Rthm} and characterized undirected graphs admitting a $k$-arc-connected orientation.

\begin{theo}[Nash-Williams \cite{NW}]\label{NWthm}
An undirected graph admits a $k$-arc-connected orientation if and only if it is $2k$-edge-connected.
\end{theo}

The following result of Ito et al. \cite{9jap} is closely related to Corollary \ref{k2cor}. They worked out a new technique how to augment arc-connectivity of a directed graph by reorienting arcs. Their result gives an algorithmic proof of Theorem \ref{NWthm}.

\begin{theo}[Ito et al. \cite{9jap}]\label{mainthm}
Let   $G=(V,E)$ be a $2k$-edge-connected undirected graph for some $k\in \mathbb{Z}_+$, and $D_0=D$ an orientation of $G.$ Then there exist orientations $D_1, \dots,D_\ell$ of $G$ such that $\ell\le (k-\lambda(D))|V|^3$, $D_{i+1}$ is obtained from $D_i$ by reversing one arc for all $0\le i\le \ell-1$ and $\lambda(D)\le\lambda(D_1)\le \dots\le\lambda(D_\ell)=k.$
\end{theo}

We mention that Theorem \ref{mainthm} has recently been extended to hypergraphs by M\"uhlenthaler, Peyrille, and Szigeti in \cite{mpsz}.
\medskip

We will need a result on minimally $k$-arc-connected digraphs.

\begin{theo}[Dalmazzo \cite{dal}]\label{Dalmazzothm}
Every $k$-arc-connected digraph $D=(V,A)$ has a spanning $k$-arc-connected digraph $D'=(V,A')$ such that $|A'|\le 2k(|V|-1)$.
\end{theo}

\section{New results and their proofs}

In this section we provide our results and their proofs.
\medskip

We first generalize  Corollary \ref{k2cor}.

\begin{theo}\label{mainthmnew}
Let $D_0$ be a directed graph,  $k\ge 2$ an integer, and $F\subseteq A$ such that  the digraph obtained from $D_0$ by reversing the arcs of $F$ is $k$-arc-connected. Then there exist reorientations $D_1, \dots,D_\ell$ of $D_0$ such that $D_{i+1}$ is obtained from $D_i$ by reversing one arc of $F$ for all $0\le i\le \ell-1$, $\lambda(D)\le\lambda(D_1)\le \dots\le\lambda(D_\ell)$ and $\lambda(D_\ell)\ge \lfloor\frac{k+3}{2}\rfloor.$ 
 \end{theo}

We obtain in the following corollary of Theorem \ref{mainthmnew} an improvement of Theorem \ref{mainthm} for $k\in \{2,3\}$, namely  $k$-arc-connectivity  can be reached by reversing at most $k(|V|-1)$ arcs.

\begin{cor}\label{maincornew}
Let   $G=(V,E)$ be a $2k$-edge-connected undirected graph for  $k\in\{2,3\}$, and $D_0=(V,A)$ an orientation of $G.$ Then there exist orientations $D_1, \dots,D_\ell$ of $G$ such that $\ell\le k(|V|-1)$, $D_{i+1}$ is obtained from $D_i$ by reversing one arc for all $0\le i\le \ell-1$, $\lambda(D_0)\le\lambda(D_1)\le \dots\le\lambda(D_\ell)=k.$ 
 \end{cor}

\begin{proof} (of Corollary \ref{maincornew})
By Theorem \ref{NWthm}, there exists a $k$-arc-connected orientation $D'$ of $G.$ By Theorem \ref{Dalmazzothm}, there exists a spanning $k$-arc-connected subgraph $D''=(V,A'')$ of $D'$ containing at most $2k(|V|-1)$ arcs. Note that $\cev{D''}$ is also a $k$-arc-connected digraph. Since $|A\cap A''|+|A''-A|=|A''|\le 2k(|V|-1),$ one of $A\cap A''$ and $A''-A$, let us denote it by $F$, contains at most $k(|V|-1)$ arcs.
Since the digraph obtained from $D$ by reversing the arcs or the reversed arcs of $F$ in $D$ contains $\cev{D''}$ or $D''$ as a spanning subgraph, it is $k$-arc-connected. We can hence apply Theorem \ref{mainthmnew} to find the required sequence of orientations since for $k\in\{2,3\}$, $\lfloor\frac{k+3}{2}\rfloor=k.$
\end{proof}
\medskip

Theorem \ref{mainthmnew} will follow from the following result that is the main contribution of the present paper. 

\begin{theo}\label{newmainthm}
Let $D=(V,A)$ be a digraph,  $k\ge 2$ an integer  and $F\subseteq A$ such that $\lambda(D)\le \lfloor\frac{k+1}{2}\rfloor$ and the digraph obtained from $D$ by reversing the arcs of $F$ is $k$-arc-connected. Then there exists an arc $\vec a$ of $F$ such that $\lambda(D)\le\lambda(D({\cev{a}})).$
 \end{theo}

Theorem \ref{newmainthm} easily implies Theorem \ref{mainthmnew}. 
Indeed, while $\lambda(D_i)\le \lfloor\frac{k+1}{2}\rfloor$, we can find, by Theorem \ref{newmainthm}, an arc $\vec{a}_i$ of $F$ such that for the digraph $D_{i+1}=D_i({\cev{a_i}})$, we have  $\lambda(D_i)\le\lambda(D_{i+1})$. Since for $k\ge 2,$ we have $\lfloor\frac{k+3}{2}\rfloor\le k$, we can hence find the required sequence of orientations.
\newpage

We now prove our main result.

\begin{proof} (of Theorem \ref{newmainthm})
First note that since $k\ge 2,$ we have 
\begin{equation}\label{23k}
	\lambda(D)\le\lfloor\frac{k+1}{2}\rfloor\le\frac{2k}{3}<k. 
\end{equation}
It follows that $F\neq\emptyset.$ Suppose for a contradiction that for all $\vec a\in F$, $\lambda(D)>\lambda(D({\cev{a}})).$ Then $\lambda(D)\ge 1.$

\begin{claim}\label{intight}
For all $\vec a\in F$, $\vec a$ enters an in-tight set in $D.$
\end{claim}

\begin{proof}
Let $\vec a\in F$ and  $X\subseteq V$ with 
$\lambda(D(\cev{a}))=d_{D({\tiny\cev{a}})}^-(X)$. 
Since $\lambda(D)-1\ge\lambda(D(\cev{a}))=d_{D({\tiny\cev{a}})}^-(X)\ge d_{D}^-(X)-1\ge \lambda(D)-1$, equality holds everywhere so $\vec{a}$ enters $X$ and $X$ is in-tight in $D.$
\end{proof}

By Claim \ref{intight}, there exists a set {\boldmath$\mathcal{L}$} of in-tight sets in $D$ such that every arc of $F$ enters an element of $\mathcal{L}$ and such that  $\sum_{Z \in \mathcal{L}} |Z||V-Z|$ is minimum. Let us fix an arbitrary vertex $s$ in $V$ and introduce the following sets: 
\begin{eqnarray*}
	\text{{\boldmath$\Fin$}} 	&=&\{ X \subseteq V - s, X \in \mathcal{L} \},\\
	\text{{\boldmath$\Fout$}} 	&=&\{ \overline X : s \in X \in \mathcal{L} \}, \\
	\text{{\boldmath$\Fc$}} 	&=&\Fin \cup \Fout.
\end{eqnarray*}

We summarize the main properties of $\Fin$, $\Fout$, and $\Fc$ in the following claim.

\begin{claim}\label{inequalities}
The following hold for $\Fin$, $\Fout$, and $\Fc$.
\begin{eqnarray}
&&\text{$\Fc$ is a non-empty laminar family,}\label{laminar}\\
&&\text{Every arc of $F$ enters an element of $\Fin$ or leaves an element of $\Fout$},\ \ \	 \label{4}\\
&&\dind{{X}}=\lambda(D) \text{ for all } \ {X} \in \Fin,\label{2}\\
&&1\le k-\lambda(D)\le\dout{X}-\din{X}\le\dout{X}-1 \text{ for all } \ {X} \in \Fin,\label{2'}\\
&&\doutd{{Y}}=\lambda(D) \hskip .13truecm \text{ for all } \ {Y} \in\Fout,	\label{3}\\
&&1\le k-\lambda(D)\le\din{{Y}}-\dout{Y}\le\din{{Y}}-1 \hskip .13truecm \text{ for all } \ {Y} \in\Fout,	\label{3'}
\end{eqnarray}
\end{claim}

\begin{proof}
Since $F\neq\emptyset$, we have $\Fc\neq\emptyset$. Applying the well-known uncrossing technique, the minimality of $\mathcal{L}$ and Theorem 2.1 of \cite{schrijverbook}, it is easy to see  that $\Fc$ is a laminar family, so \eqref{laminar} holds.  

Let $a\in F$. Then $a$ enters a set $Z \in \mathcal{L}$. If $s \in Z$ then $a$ leaves $\overline Z \in \Fout$, if $s \not \in Z$ then $a$ enters $Z \in \Fin$, hence \eqref{4} holds. 

Let $X \in \mathcal{L}$. Then $\dind{X}=\lambda(D)$ and  \eqref{2} holds. 
By the minimality of $\mathcal{L}$,  $\din{X} \ge 1$. Then, by  $\lambda(D)<k\le\lambda(D(\cev{F})),$ we have $1\le k-\lambda(D)\le d_{D({\tiny \cev{F}})}^-(X) - \dind{X} = \dout{X} - \din{X}$ and  \eqref{2'} holds. 
If $X \in \Fin$ then $X \in \mathcal{L}$, so \eqref{2} and  \eqref{2'} hold. 
If $Y\in\Fout$ then $\overline Y\in \mathcal{L}$, so $\overline Y$ satisfies \eqref{2} and  \eqref{2'}, and hence $Y$ satisfies \eqref{3} and  \eqref{3'}.
\end{proof}

For all $Z \in \Fc\cup\{V\}$, we introduce the following notations.
 \begin{eqnarray*}
\text{\boldmath$\Uin{Z}$} 		&	=	& 	\bigcup\{X \in\Fin:X\subsetneq Z\},\\
\text{\boldmath$\Uout{Z}$} 	& 	=	&	\bigcup\{Y\in \Fout:Y\subsetneq Z\},\\
\text{\boldmath$\E{Z}$} 		&	= 	&	\text{the maximal elements of $\Fc$ strictly contained in $Z$,}\\
\text{\boldmath$\Ein{Z}$} 		&	= 	&	\E{Z} \cap \Fin,\\
\text{\boldmath$\Eout{Z}$} 	& 	= 	& 	\E{Z} \cap \Fout,\\
\text{\boldmath$M_-(Z)$} 		&	=	&	\bigcup\{X\in\Ein{Z}\},\\
\text{\boldmath$M_+(Z)$} 		&	=	& 	\bigcup\{X\in\Eout{Z}\}.
\end{eqnarray*}

The following lemma will easily imply the theorem.

\begin{lemma}\label{jbdievu}
The following hold for all $X\in\Fin\cup\{V\}$ and $Y\in\Fout\cup\{V\}$.
	\begin{eqnarray}
		\delta_F(X)\cap\delta_F(O(X))	&	=		&\emptyset,\label{x1}			\\
		\rho_F(M_+(X))				&	\subseteq 	&\rho_F({X}),\label{x2}			\\
		\din{M_+(X)}-\dout{M_+(X)}	&	\ge 		&(k-\lambda(D))|\Eout{X}|,\label{x3}	\\
		\rho_F(Y)\cap\rho_F(I(Y))		&	=		&\emptyset,\label{y1}			\\
		\delta_F(M_-(Y))			&	\subseteq	& \delta_F(Y),\label{y2}			\\
		\dout{M_-({Y})}-\din{M_-({Y})}	&	\ge 		&(k-\lambda(D))|\Ein{Y}|.   \label{y3}
	\end{eqnarray}
\end{lemma}

\begin{proof}
Suppose for a contradiction that the lemma does not hold and let $Z$ be a minimal set violating one of the conditions \eqref{x1}--\eqref{y3}. We may suppose that $Z=Y\in\Fout\cup\{V\}$ because the situation is completely symmetric. 
Let $S= M_-({Y})$ and  $T= {Y} - S.$ 
\begin{claim}\label{y2holds}
\eqref{y2} holds for $Y.$
\end{claim}

\begin{proof}
Suppose for a contradiction that there exists $uv\in\delta_F(S)\setminus\delta_F(Y),$ that is $u\in S$ and $v\in T.$ Then $uv$ leaves a set $X\in\Ein{{Y}}.$ Since $X\in\Fin$ and  $X\subsetneq Y,$ the minimality of $Y$ implies that $X$ satisfies \eqref{x1}, so $uv$ does not leave $\Uout{X}.$  Then, by \eqref{laminar}, $uv $ does not leave any element of $\Fout$, so, by \eqref{4}, $uv$ enters an element $X'$ of $\Fin.$ By \eqref{laminar} and $v\in T,$ we have $X'\subseteq Y-M_-({Y}).$ Since $X'\in\Fin$ but $X'\cap M_-({Y})=\emptyset$,  $X'$ is contained in a set $Y'\in \Eout{{Y}}.$ Since $u\in X,$ $v\in X'\subseteq\Uin{Y'}\subseteq Y'$ and $X\cap Y'=\emptyset$,  we get that $uv\in\rho_F(Y')\cap\rho_F(I(Y'))$ that is \eqref{y2} does not hold for $Y'.$ This, by $Y'\subsetneq Y$, contradicts the minimality of $Y$ and the claim follows. 
\end{proof}
\begin{claim}\label{y3holds}
\eqref{y3} holds for $Y.$
\end{claim}

\begin{proof}
By \eqref{laminar},  $\Ein{Y}$ is a partition of $S$. Thus, by  \eqref{2'}, we get 
\begin{align*}
	\dout{S}-\din{S}		&	=	\sum_{X\in\Ein{{Y}}}(\dout{X}-\din{X})		\\
					&	\ge 	(k-\lambda(D))|\Ein{Y}|
\end{align*}

and the claim follows.
\end{proof}

By Claims \ref{y2holds} and \ref{y3holds} and the definition of $Y$, it follows that $Y$ violates \eqref{y1}, that is $1\le d_F(\overline{Y},I(Y)).$ 
Note that, by \eqref{laminar}, we have  $I(Y)\subseteq M_-(Y)\cup\bigcup_{Y_i\in\Eout{Y}} \Uin{Y_i}.$
Note also that $Y_i\subsetneq Y$ and so $\overline{Y}\subsetneq \overline{Y_i}$ for all $Y_i\in\Eout{Y}$.
Hence every arc that goes from $\overline{Y}$ to $I(Y)$ either goes from $\overline{Y}$ to $M_-(Y)$ or from $\overline{Y_i}$ to $\Uin{Y_i}$ for some $Y_i\in\Eout{Y}$.
By the minimality of $Y$,  $Y_i$ satisfies \eqref{y1}, thus no arc goes from $\overline{Y_i}$ to $\Uin{Y_i}$ for any $Y_i\in\Eout{Y}$.
It follows that
\begin{align*}
1	&	\le 	d_F(\overline{Y},I(Y))\\
	&	\le 	d_F(\overline{Y},M_-(Y))+\sum_{Y_i\in\Eout{Y}} d_F(\overline{Y_i},\Uin{Y_i})\\
	&	= 	d_F(\overline{Y},M_-(Y)).
	\end{align*} 

 Hence $|\Ein{Y}|\ge 1.$ By \eqref{3}, Claims \ref{y2holds} and \ref{y3holds}, \eqref{2'}, we get 
\begin{align}
	\lambda(D)	&	=	\doutd{Y}						\nonumber	\\
				&	\ge	\dout{S}+d_A(T,\overline Y)		\nonumber	\\
				&	\ge	\din{S}+(k-\lambda(D))|\Ein{Y}|+0	\nonumber	\\
				&	\ge	1+(k-\lambda(D))|\Ein{Y}|.\label{jhvhk}
\end{align} 

If $|\Ein{Y}|\ge 2$ then, by \eqref{jhvhk} and \eqref{23k}, we get  the following  a contradiction.
\begin{align*}
	\lambda(D)	&	\ge 	1+2(k-\lambda(D))						\\
				&	\ge 	1+2(\frac{3}{2}\lambda(D)-\lambda(D))		\\
				&	=	1+\lambda(D).
\end{align*}

Thus, $\Ein{Y}=\{X\},$ and then, by \eqref{jhvhk} and \eqref{23k}, we have 
\begin{equation*}
	\lambda(D)\ge 1+k-\lambda(D)\ge\lambda(D).
\end{equation*}

 Hence equality holds everywhere in \eqref{jhvhk}, so  $d_A(T,\overline Y)=0$. 
  
 We claim that $T\neq\emptyset.$ Otherwise, by \eqref{3'} (if $Y\neq V$), $Y=S$, Claim \ref{y3holds}, \eqref{3'} and $|\Ein{Y}|=1$, we have the following contradiction.
\begin{equation*}
 0\ge\dout{Y}-\din{Y}=\dout{S}-\din{S}>0.
\end{equation*}

Finally,  by \eqref{2}, $d_A(T,\overline Y)=0$, $d_F(\overline{Y},M_-(Y))\ge 1$, $T\neq\emptyset,$ and $\lambda(D)$-arc-connectivity of $D$, we have the following contradiction.
\begin{align*}
\lambda(D)	&	=	d_A^-(X)+0								\\
			&	\ge 	d_F(\overline Y,X)+d_A(T,X)+d_A(T,\overline Y)		\\
			&	\ge 	1+d_A^+(T)								\\
			&	\ge 	1+\lambda(D).
\end{align*} 
The proof of the lemma is complete.
\end{proof}

By \eqref{laminar},  we have  $\mathcal{M}(V)\neq\emptyset,$ say $\Ein{V}\neq\emptyset.$ We conclude by applying Lemma \ref{jbdievu} for $V.$ By \eqref{y2}, we get $\delta_F(M_-(V))\subseteq\delta_F(V)=\emptyset,$ and then, by \eqref{y3}, \eqref{3'} and $\mathcal{M}(V)\neq\emptyset$, we have  the following contradiction.
\begin{align*}
0 \ge \dout{{{M_-}(V)}} - \din{{M_-}(V)} > 0.
\end{align*} 

This finishes the proof of the theorem.
\end{proof}

\section{Examples}

We first provide an example of a $3$-arc-connected digraph $D_1 = (V, A)$ with an arc set $F$ in $A$ such that reversing the arcs of $F$ we obtain a digraph which is $4$-arc-connected but no  arc of $F$ can be reversed without destroying $3$-arc-connectivity. Let us consider the digraph $D_1$ of Figure \ref{fig:lD3}, where a value on an arc means its multiplicity, an arc without value is of multiplicity $1$,  and the arcs of $F$ are in blue. Note that $D_1$ is $3$-arc-connected and, by reversing the blue arcs of $D_1,$ we obtain a digraph $D_1'$ which is $4$-arc-connected. The sets $\Fin$ and $\Fout$ of the proof of Theorem \ref{newmainthm} are represented by red and green circles respectively. These sets show that no blue arc can be reversed without destroying $3$-arc-connectivity. Note that $\lambda(D_1)=3>\lfloor\frac{4+1}{2}\rfloor.$ This example shows that the condition $\lambda(D_1)\le\lfloor\frac{k+1}{2}\rfloor$ in Theorem \ref{newmainthm} can not be deleted.

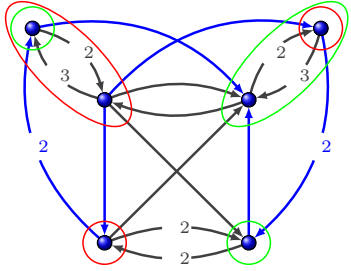
\begin{figure}[!h]
    \centering
    \resizebox{.4\linewidth}{!}{
    \begin{tikzpicture}
        \Vertex[x=0, y=6, size=0.4, style={shading=ball}]{A}
        \Vertex[x=8, y=6, size=0.4, style={shading=ball}]{B}
        \Vertex[x=2, y=4, size=0.4, style={shading=ball}]{C}
        \Vertex[x=6, y=4, size=0.4, style={shading=ball}]{D}
        %\Vertex[x=4, y=2, size=0.4, style={shading=ball}]{E}
        \Vertex[x=2, y=0, size=0.4, style={shading=ball}]{F}
        \Vertex[x=6, y=0, size=0.4, style={shading=ball}]{G}
   
         \Edge[lw=2pt, label=$2$, bend=30, Direct, fontscale=2](A)(C)
        \Edge[lw=2pt, label=$3$, bend=30, Direct, fontscale=2](C)(A)
        \Edge[lw=2pt, label=$3$, bend=30, Direct, fontscale=2](B)(D)
        \Edge[lw=2pt, label=$2$, bend=30, Direct, fontscale=2](D)(B)
        
        \Edge[lw=2pt,  bend=20, Direct, fontscale=2](C)(D)
        \Edge[lw=2pt,  bend=20, Direct, fontscale=2](D)(C)
        
        \Edge[lw=2pt,  bend=30, Direct, fontscale=2, color=blue](C)(B)
        \Edge[lw=2pt,  bend=30, Direct, fontscale=2, color=blue](A)(D)
        \Edge[lw=2pt, label=$2$, bend=30, Direct, fontscale=2, color=blue](F)(A)
        \Edge[lw=2pt, label=$2$, bend=30, Direct, fontscale=2, color=blue](B)(G)

%        \Edge[lw=2pt, label=$2$, bend=30, Direct, fontscale=2](C)(E)
%        \Edge[lw=2pt, label=$1$, bend=30, Direct, fontscale=2](E)(C)
%        \Edge[lw=2pt, label=$1$, bend=30, Direct, fontscale=2](D)(E)
%        \Edge[lw=2pt, label=$2$, bend=30, Direct, fontscale=2](E)(D)
%        
        \Edge[lw=2pt,  Direct, fontscale=2](C)(G)
        \Edge[lw=2pt,  Direct, fontscale=2](F)(D)
        \Edge[lw=2pt,  Direct, fontscale=2, color=blue](C)(F)
        \Edge[lw=2pt,  Direct, fontscale=2, color=blue](G)(D)
 
%        \Edge[lw=2pt, label=$3$, Direct, fontscale=2](F)(E)
%        \Edge[lw=2pt, label=$3$, Direct, fontscale=2](E)(G)
        \Edge[lw=2pt, label=$2$, bend=20, Direct, fontscale=2](G)(F)
        \Edge[lw=2pt, label=$2$, bend=20, Direct, fontscale=2](F)(G)
        
        \draw[color=red, rotate around={45:(1,5)}, very thick](1,5) ellipse (0.9 and 2.3);
        \draw[color=green, rotate around={135:(7,5)}, very thick](7,5) ellipse (0.9 and 2.3);
        
        \draw[color=red, very thick](2,0) circle (0.6);
        \draw[color=red, very thick](8,6) circle (0.6);
        \draw[color=green, very thick](0,6) circle (0.6);
        \draw[color=green, very thick](6,0) circle (0.6);
    
    \end{tikzpicture}
    }
    \caption{Digraph $D_1$ with $\lambda(D_1) = 3$, arcs of $F$ in blue, and $\lambda(D_1') = 4.$}
    \label{fig:lD3}
\end{figure}

Now we provide a more complicated example that shows that the condition $\lambda(D)\le\lfloor\frac{k+1}{2}\rfloor$ in Theorem \ref{newmainthm} can not be replaced by $\lambda(D)\le\lfloor\frac{2k}{3}\rfloor.$ Let us consider the digraph $D_2$ of Figure \ref{fig:lD4}, where as before a value on an arc means its multiplicity, an arc without value is of multiplicity $1$,  and the arcs of $F$ are in blue. Note that $D_2$ is $4$-arc-connected and, by reversing the blue arcs of $D_2,$ we obtain a digraph $D_2'$ which is $6$-arc-connected. The sets $\Fin$ and $\Fout$ of the proof of Theorem \ref{newmainthm} are represented by red and green circles respectively. These sets show that no blue arc can be reversed without destroying $4$-arc-connectivity. Note that $\lambda(D_2)=4=\frac{2\cdot 6}{3}.$
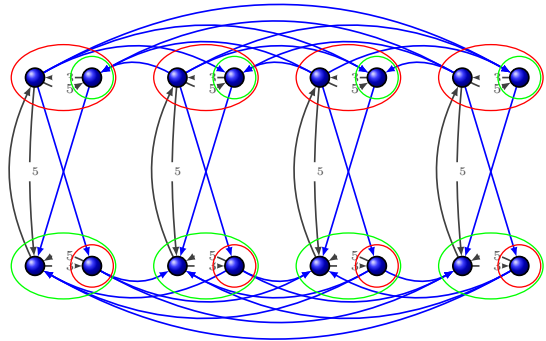
\begin{figure}[!h]
    \centering
    \resizebox{0.6\linewidth}{!}{
    \begin{tikzpicture}
        %==========RED==========
        \draw[color=red, thick](0.6,6) ellipse (1.1 and 0.7);
        \Vertex[x=0, y=6, size=0.4, style={shading=ball}]{Au}
        \Vertex[x=1.2, y=6, size=0.4, style={shading=ball}]{Ac}
        \draw[color=green, thick](1.2,6) circle (0.45);

        \draw[color=red, thick](3.6,6) ellipse (1.1 and 0.7);
        \Vertex[x=3, y=6, size=0.4, style={shading=ball}]{Bu}
        \Vertex[x=4.2, y=6, size=0.4, style={shading=ball}]{Bc}
        \draw[color=green, thick](4.2,6) circle ( 0.45);

        \draw[color=red, thick](6.6,6) ellipse (1.1 and 0.7);
        \Vertex[x=6, y=6, size=0.4, style={shading=ball}]{Cu}
        \Vertex[x=7.2, y=6, size=0.4, style={shading=ball}]{Cc}
        \draw[color=green, thick](7.2,6) circle ( 0.45);

        \draw[color=red, thick](9.6,6) ellipse (1.1 and 0.7);
        \Vertex[x=9, y=6, size=0.4, style={shading=ball}]{Du}
        \Vertex[x=10.2, y=6, size=0.4, style={shading=ball}]{Dc}
        \draw[color=green, thick](10.2,6) circle ( 0.45);

        %==========GREEN==========

        \draw[color=green, thick](0.6,2) ellipse (1.1 and 0.7);
        \Vertex[x=0, y=2, size=0.4, style={shading=ball}]{Eu}
        \Vertex[x=1.2, y=2, size=0.4, style={shading=ball}]{Ec}
        \draw[color=red, thick](1.2,2) circle ( 0.45);

        \draw[color=green, thick](3.6,2) ellipse (1.1 and 0.7);
        \Vertex[x=3, y=2, size=0.4, style={shading=ball}]{Fu}
        \Vertex[x=4.2, y=2, size=0.4, style={shading=ball}]{Fc}
        \draw[color=red, thick](4.2,2) circle ( 0.45);

        \draw[color=green, thick](6.6,2) ellipse (1.1 and 0.7);
        \Vertex[x=6, y=2, size=0.4, style={shading=ball}]{Gu}
        \Vertex[x=7.2, y=2, size=0.4, style={shading=ball}]{Gc}
        \draw[color=red, thick](7.2,2) circle ( 0.45);

        \draw[color=green, thick](9.6,2) ellipse (1.1 and 0.7);
        \Vertex[x=9, y=2, size=0.4, style={shading=ball}]{Hu}
        \Vertex[x=10.2, y=2, size=0.4, style={shading=ball}]{Hc}
        \draw[color=red, thick](10.2,2) circle ( 0.45);

        %=====================ARCS=====================
        \Edge[lw=1pt, Direct, color=blue](Au)(Ec)
        \Edge[lw=1pt, bend=-5, Direct, label=$5$, fontscale=1](Au)(Eu)
        \Edge[lw=1pt, bend=25, Direct,  fontscale=1](Eu)(Au)
        \Edge[lw=1pt, Direct, color=blue](Ac)(Eu)

        \Edge[lw=1pt, Direct, color=blue](Bu)(Fc)
        \Edge[lw=1pt, bend=-5, Direct, label=$5$, fontscale=1](Bu)(Fu)
        \Edge[lw=1pt, bend=25, Direct,  fontscale=1](Fu)(Bu)
        \Edge[lw=1pt, Direct, color=blue](Bc)(Fu)

        \Edge[lw=1pt, Direct, color=blue](Cu)(Gc)
        \Edge[lw=1pt, bend=-5, Direct, label=$5$, fontscale=1](Cu)(Gu)
        \Edge[lw=1pt, bend=25, Direct,  fontscale=1](Gu)(Cu)
        \Edge[lw=1pt, Direct, color=blue](Cc)(Gu)

        \Edge[lw=1pt, Direct, color=blue](Du)(Hc)
        \Edge[lw=1pt, bend=-5, Direct, label=$5$, fontscale=1](Du)(Hu)
        \Edge[lw=1pt, bend=25, Direct,  fontscale=1](Hu)(Du)
        \Edge[lw=1pt, Direct, color=blue](Dc)(Hu)

        \Edge[lw=1pt, bend=30, Direct, color=blue](Au)(Bc)
        \Edge[lw=1pt, bend=30, Direct, color=blue](Au)(Cc)
        \Edge[lw=1pt, bend=30, Direct, color=blue](Au)(Dc)

        \Edge[lw=1pt, bend=-30, Direct, color=blue](Bu)(Ac)
        \Edge[lw=1pt, bend=30, Direct, color=blue](Bu)(Cc)
        \Edge[lw=1pt, bend=30, Direct, color=blue](Bu)(Dc)

        \Edge[lw=1pt, bend=-30, Direct, color=blue](Cu)(Bc)
        \Edge[lw=1pt, bend=-30, Direct, color=blue](Cu)(Ac)
        \Edge[lw=1pt, bend=30, Direct, color=blue](Cu)(Dc)

        \Edge[lw=1pt, bend=-30, Direct, color=blue](Du)(Bc)
        \Edge[lw=1pt, bend=-30, Direct, color=blue](Du)(Cc)
        \Edge[lw=1pt, bend=-30, Direct, color=blue](Du)(Ac)

        \Edge[lw=1pt, bend=-30, Direct, color=blue](Ec)(Fu)
        \Edge[lw=1pt, bend=-30, Direct, color=blue](Ec)(Gu)
        \Edge[lw=1pt, bend=-30, Direct, color=blue](Ec)(Hu)

        \Edge[lw=1pt, bend=30, Direct, color=blue](Fc)(Eu)
        \Edge[lw=1pt, bend=-30, Direct, color=blue](Fc)(Gu)
        \Edge[lw=1pt, bend=-30, Direct, color=blue](Fc)(Hu)

        \Edge[lw=1pt, bend=30, Direct, color=blue](Gc)(Fu)
        \Edge[lw=1pt, bend=30, Direct, color=blue](Gc)(Eu)
        \Edge[lw=1pt, bend=-30, Direct, color=blue](Gc)(Hu)

        \Edge[lw=1pt, bend=30, Direct, color=blue](Hc)(Fu)
        \Edge[lw=1pt, bend=30, Direct, color=blue](Hc)(Gu)
        \Edge[lw=1pt, bend=30, Direct, color=blue](Hc)(Eu)

        \Edge[lw=1pt, label=$3$, Direct, fontscale=1.18 ](Ac)(Au)
        \Edge[lw=1pt, label=$5$, bend=-30, Direct, fontscale=1.18 ](Au)(Ac)

        \Edge[lw=1pt, label=$3$, Direct, fontscale=1.18 ](Bc)(Bu)
        \Edge[lw=1pt, label=$5$, bend=-30, Direct, fontscale=1.18 ](Bu)(Bc)

        \Edge[lw=1pt, label=$3$, Direct, fontscale=1.18 ](Cc)(Cu)
        \Edge[lw=1pt, label=$5$, bend=-30, Direct, fontscale=1.18 ](Cu)(Cc)

        \Edge[lw=1pt, label=$3$, Direct, fontscale=1.18 ](Dc)(Du)
        \Edge[lw=1pt, label=$5$, bend=-30, Direct, fontscale=1.18 ](Du)(Dc)

        \Edge[lw=1pt, label=$3$, Direct, fontscale=1.18 ](Eu)(Ec)
        \Edge[lw=1pt, label=$5$, bend=-30, Direct, fontscale=1.18 ](Ec)(Eu)

        \Edge[lw=1pt, label=$3$, Direct, fontscale=1.18 ](Fu)(Fc)
        \Edge[lw=1pt, label=$5$, bend=-30, Direct, fontscale=1.18 ](Fc)(Fu)

        \Edge[lw=1pt, label=$3$, Direct, fontscale=1.18 ](Gu)(Gc)
        \Edge[lw=1pt, label=$5$, bend=-30, Direct, fontscale=1.18 ](Gc)(Gu)

        \Edge[lw=1pt, label=$3$, Direct, fontscale=1.18 ](Hu)(Hc)
        \Edge[lw=1pt, label=$5$, bend=-30, Direct, fontscale=1.18 ](Hc)(Hu)
    \end{tikzpicture}
    }
    \caption{Digraph $D_2$ with $\lambda(D_2) = 4$, arcs of $F$ in blue, and $\lambda(D_2') = 6.$}
    \label{fig:lD4}
\end{figure}

\end{document}